\documentclass[11pt,reqno]{amsart}
\usepackage{graphicx}
\usepackage{cite}
\usepackage{amscd}
\usepackage{amsmath}
\usepackage{amsfonts}
\usepackage{setspace}
\usepackage{amssymb}
\newtheorem{theorem}{Theorem}[section]
\newtheorem{corollary}{Corollary}[section]
\newtheorem{lemma}{Lemma}[section]
\newtheorem{definition}{Definition}[section]
\newtheorem{remark}{Remark}[section]
\newtheorem{proposition}{Proposition}[section]
\newtheorem{example}{Example}[section]
\newcommand{\R}{{\mathbb R}}

\newcommand{\lm}{\lambda}

\newcommand{\N}{{\mathbb N}}

\newcommand{\di}{\mbox{div}}
\newcommand{\al}{\alpha}
\newcommand{\be}{\beta}
\newcommand{\ta}{\theta}
\newcommand{\ga}{\gamma}
\newcommand{\na}{\nabla}

\newcommand{\Ga}{\Gamma}
\newcommand{\ep}{\epsilon}
\newcommand{\Om}{\Omega}
\newcommand{\pom}{\partial\Omega}
\newcommand{\la}{\langle}
\newcommand{\ra}{\rangle}
\newcommand{\de}{\delta}

\newcommand{\rhp}{\rightharpoonup}
\newcommand{\hra}{\hookrightarrow }
\DeclareMathOperator{\meas}{meas}

\newcommand{\ov}{\overline}

\newcounter{romnum}

\newcounter{arabicnum}

\begin{document}
\title[Weak solutions]
{A degenerate elliptic system with variable exponents}

\author[L. Kong]{Lingju Kong}
\address{Department of Mathematics,
University of Tennessee at Chattanooga, Chattanooga, TN 37403, USA.}
\email{Lingju-Kong@utc.edu}

\baselineskip 18pt

\begin{abstract}
We study a degenerate elliptic system with variable exponents.
Using the variational approach and some recent theory on weighted Lebesgue and Sobolev spaces
with variable exponents, we prove the existence of at least two distinct nontrivial
weak solutions of the system. Several consequences of the main theorem are derived; in particular, the existence of at lease two
distinct nontrivial nonnegative solution are established for a scalar degenerate problem. One example is provided to show
the applicability of our results.
\end{abstract}

\maketitle

\setcounter{equation}{0}
\section{Introduction}
\label{sec1}

In this paper, we are concerned with the existence of nontrivial weak solutions of the 
elliptic system with degenerate $p_i(x)$-Laplacian operators
\begin{equation}\label{1.1}
\left\{\begin{array}{l}
-\di(w_i(x)|\na u_i|^{p_i(x)-2}\na u_i)=\lm f_i(x,u_1,\ldots,u_n)\quad  \text{in}\  \Om,\ i=1,\ldots,n, \\\noalign{\medskip}
u_i=0\quad \text{on}\  \pom,\ i=1,\ldots,n,
\end{array}
\right.
\end{equation}
where $n,N\in\N$, $\Om\subset\R^N$ is a bounded domain with a Lipschitz boundary $\pom$,
$\lm>0$ is a parameter, $p_i\in C_+(\ov{\Om}):=C(\ov{\Om},(1,\infty))$, $w_i$ are weight functions defined in $\Om$, i.e., functions 
measurable and positive a.e. in $\Om$, and $f_i\in C(\Om\times\R^n,\R)$ are such that that there exists a function
$F\in C^1({\Om}\times \R^n,\R)$ such that 
$\na F(x,t_1,\ldots,t_n)=(f_1(x, t_1,\ldots,t_n),\ldots,f_n(x, t_1,\ldots,t_n))$ in $\Om\times\R^n$.
Here, the operators $\di(w_i(x)|\na u_i|^{p_i(x)-2}\na u_i)$, $i=1,\ldots,n$, are called degenerate $p(x)$-Laplacian operators,
and as usual, when $w_i(x)\equiv 1$, they are called $p(x)$-Laplacian operators.

The degeneracy or singularity of system \eqref{1.1} is considered in the case that the weight functions $w_i$, $i=1,\ldots, n$, 
are allowed to be unbounded and/or not separated from zero. The character of the operator $\di(w_i(x)|\na u_i|^{p_i(x)-2}\na u_i)$ can be 
interpreted as a degeneration or as a singularity of $\di(\na u_i|^{p_i(x)-2}\na u_i)$. 
As is well known, degenerate phenomena occur frequently in many areas (\!\!\cite{dd, dl}).
Degenerated quasilinear elliptic equations with $p$-Laplacian were extensively studied in the
1990s and the related results were summarized in the monograph \cite{dkn}.
On the other hand, differential equations and variational problems with variable exponents 
have applications in mathematical physics (\!\!\cite{as, clr, mr, vz}).
Moreover, in recent years, degenerate elliptic problems with variable exponents have attracted the attention of several researchers
and many papers have been published to study these problems. See,
for example, \cite{HS1, HS2, kwz, ki, MR}. We point out that the study of these problems relies heavily on the
theory for weighted variable exponent Lebesgue and Sobolev spaces, which is briefly reviewed in Section \ref{sec2}.

In the literature, many results have been obtained for various variations of the scalar case of system \eqref{1.1},
and in the following, we just mention a few of them. 
Mih\u{a}ilescu and R\u{a}dulescu \cite{MR} studied the existence of at least two nontrivial nonnegative weak solutions 
for the problem
\begin{equation}\label{1.2}
\left\{\begin{array}{l}
-\di(a(x,\nabla u))=\lm \left(u^{\gamma-1}-u^{\beta-1}\right)\quad \text{in}\  \Om,\  \\\noalign{\medskip}
u=0\quad \text{on}\  \pom,
\end{array}
\right.
\end{equation}
where $a : \ov{\Om}\times\R^N\to\R^N$ is continuous and satisfies, among others, the condition that there exist $p\in C_+(\ov{\Om})$ and $c_1>0$ such that 
$|a(x, \xi)|\leq c_1\left(1+|\xi|^{p(x)-1}\right)$ for all $x\in\ov{\Om}$ and $\xi\in\R^N$
and $1<\beta<\gamma<\min_{x\in\ov{\Om}}p(x)$.
Kim et al. \cite{kwz} studied the global behavior of the set of solutions for the problem
\begin{equation}\label{1.2+}
\left\{\begin{array}{l}
-\di(w(x)|\na u|^{p(x)-2}\na u)=\mu g(x)|u|^{p(x)-2}u+f(\lm, x, u, \na u)\quad \text{in}\  \Om,\  \\\noalign{\medskip}
u=0\quad \text{on}\  \pom.
\end{array}
\right.
\end{equation}
Ho and Sim \cite{ki} considered the existence and multiplicity of weak solutions to the degenerate problem
\begin{equation}\label{1.3}
\left\{\begin{array}{l}
-\di(a(x,\nabla u))=\lm f(x,u)\quad \text{in}\  \Om,\  \\\noalign{\medskip}
u=0\quad \text{on}\  \pom.
\end{array}
\right.
\end{equation}
Especially, under some suitable conditions on $a(x,\xi)$ and $f(x,t)$, they proved the existence of two nontrivial solutions for problem \eqref{1.3}.
See \cite[Theorem 4.5]{ki}.
In another recent paper \cite{HS1}, the authors studied the existence and multiplicity of nontrivial nonnegative solutions for the problem
\begin{equation}\label{1.4}
\left\{\begin{array}{l}
-\di(w(x)|\na u|^{p(x)-2}\na u)=\lm a(x)|u|^{q(x)-2}u+\mu b(x)|u|^{h(x)-2}u\quad \text{in}\  \Om,\  \\\noalign{\medskip}
u=0\quad \text{on}\  \pom,
\end{array}
\right.
\end{equation}
where $p, q, h\in C_+(\ov{\Om})$, $w$, $a$, $b$ are measurable and positive a.e. in $\Om$, and $\lm$, $\mu$ are real parameters. 
Under some appropriate conditions on $w$, the authors proved, among others, that
for each $\mu>0$, there exists $\ov{\lm}=\ov{\lm}(\mu)>0$ such that for all $\lm\in (0,\ov{\lm})$, problem \eqref{1.4} has two nontrivial nonnegative solutions. See
\cite[Corollary 3.3]{HS1}.
R\u{a}dulescu and Repov\u{s} \cite{rr1} applied monotonicity methods and variational arguments to study the existence of positive solutions of
the problems
\begin{equation*}
\left\{\begin{array}{l}
-\Delta u=\lm k(x)u^q\pm h(x)u^q\quad \text{in}\  \Om,\  \\\noalign{\medskip}
u=0\quad \text{on}\  \pom,
\end{array}
\right.
\end{equation*}
where $h, \ k$ are nonnegative.

Motivated by these works, in this paper, we study the existence of at least two nontrivial weak solutions of system \eqref{1.1}.
We find sufficient conditions under which there exists $\lm_0>0$ such that system \eqref{1.1} has at least two distinct nontrivial weak solutions for all $\lm\in (\lm_0,\infty)$.
The proof of our main theorem is variational in nature; in particular, 
the classic mountain pass lemma of Ambrosetti and Rabinowitz is utilized in the proof.
Several corollaries of the main theorem are obtained, and in particular, one corollary establishes the existence 
of two distinct nontrivial nonnegative weak solutions for a scalar degenerate problem.
This paper extends and develops many existing ideas and results in the literature, for example, 
those in \cite{HS1, HS2,ki, MR}, to deal with the situation where the existence of multiple solutions are investigated
for a degenerate system instead of a scalar equation. 

Finally, we comment that, with little modification of the arguments, the results obtained in this paper can be extended to the problem
\begin{equation}\label{1.5}
\left\{\begin{array}{l}
-\di(a_i(x,\na u_i))=\lm f_i(x,u_1,\ldots,u_n)\quad \text{in}\  \Om,\ i=1,\ldots,n, \\\noalign{\medskip}
u_i=0\quad \text{on}\  \pom,\ i=1,\ldots,n,
\end{array}
\right.
\end{equation}
where $a_i$, $i=1,\ldots, n$, satisfy some appropriate properties. (See, for example, the conditions (A0)--(A5) in \cite{ki}). For the ease of the discussion, 
we study system \eqref{1.1} instead of \eqref{1.5} in this paper.


The rest of this paper is organized as follows. Section \ref{sec2} contains some 
preliminary results, Section \ref{sec3} contains the main results, and the proofs
of the main results are given in Section \ref{sec4}.


\setcounter{equation}{0}
\section{Preliminary results}
\label{sec2}

Let $\Om\subset\R^N$ be a bounded domain with a Lipschitz boundary $\pom$, $p\in C_+(\ov{\Om})$, and
$w(x)$ be a weight function in $\Om$.
In this section, we only review some basic results for weighted Lebesgue and Sobolev spaces
$L^{p(x)}(w, \Om)$ and $W^{1,p(x)}(w,\Om)$  with variable exponents.
These results can be found, for example, in \cite{HS1, HS2, kwz}.
For results on unweighted variable exponent Lebesgue and Sobolev spaces $L^{p(x)}(\Om)$ and $W^{1,p(x)}(\Om)$, 
we refer the reader to \cite{DHHR, FH}
and the references therein. 

The variable exponent Lebesgue space $L^{p(x)}(w, \Om)$ is defined by
\begin{eqnarray*}
L^{p(x)}(w, \Om)=\left\{u : \Om\to\R \ \text{is measurable and}\  \int_{\Om}w(x)|u(x)|^{p(x)}dx<\infty\right\}.
\end{eqnarray*}
Then, $L^{p(x)}(w, \Om)$ is a normed space equipped with the Luxemburg norm
\begin{equation*}
|u|_{L^{p(x)}(w, \Om)}=\inf\left\{\lm>0\ :\ \int_{\Om}w(x)\left|\frac{u}{\lm}\right|^{p(x)}dx\leq 1\right\}.
\end{equation*}
When $w(x)\equiv 1$ in $\Om$, we write $L^{p(x)}(w, \Om)$ as $L^{p(x)}(\Om)$ and we use the notation
$|u|_{L^{p(x)}(\Om)}$ instead of $|u|_{L^{p(x)}(w, \Om)}$.

Throughout this paper, for any $h\in C_+(\ov{\Om})$, we use the notations:
\begin{equation}\label{2.1}
h^+=\max_{x\in\Om}h(x)\quad   \text{and}\quad   h^-=\min_{x\in\Om}h(x),
\end{equation}
and let $\hat{h}(x)$ denote the conjugate of $h(x)$, i.e., $1/h(x)+1/\hat{h}(x)=1$.

The following proposition can be found in \cite[Proposition 2.4]{FH}.

\begin{proposition}\label{p2.1}
The space $L^{p(x)}(\Om)$ is a separable and uniformly convex Banach space, and its 
conjugate space is $L^{\hat{p}(x)}(\Om)$.
Moreover, for any $u\in L^{p(x)}(\Om)$ and $v\in L^{\hat{p}(x)}(\Om)$, we have the following H\"{o}lder-type inequality
\begin{equation*}
\left|\int_{\Om}uvdx\right|\leq \left(\frac{1}{p^-}+\frac{1}{\hat{p}^-}\right)|u|_{L^{p(x)}(\Om)}|v|_{L^{\hat{p}(x)}(\Om)}\leq 2|u|_{L^{p(x)}(\Om)}|v|_{L^{\hat{p}(x)}(\Om)}.
\end{equation*}

\end{proposition}

Propositions \ref{p2.2}--\ref{p2.4} are taken from \cite[Propositions \ref{p2.2}--\ref{p2.4}]{HS2}, respectively.

\begin{proposition}\label{p2.2}
Let the modular $\rho: L^{p(x)}(w, \Om)\to\R$ be defined by
$$
\rho(u)=\int_{\Om}w(x)|u|^{p(x)}dx\quad   \text{for all}\  u\in L^{p(x)}(w,\Om).
$$
Then, for all $u\in L^{p(x)}(w,\Om)$, we have

\begin{itemize}

\item[(a)] $|u|_{L^{p(x)}(w,\Om)}>1$  \mbox{(}$=1$, $<1$) $\Longleftrightarrow$ $\rho(u)>1$ ($=1$, $<1$), respectively.

\item[(b)]  $|u|_{L^{p(x)}(w,\Om)}\geq 1$ $\Longrightarrow$ $|u|_{L^{p(x)}(w,\Om)}^{p^-}\leq \rho(u)\leq |u|_{L^{p(x)}(w,\Om)}^{p^+}${;}

\item[(c)] $|u|_{L^{p(x)}(w,\Om)}\leq 1$ $\Longrightarrow$ $|u|_{L^{p(x)}(w,\Om)}^{p^+}\leq \rho(u)\leq |u|_{L^{p(x)}(w,\Om)}^{p^-}$;

\item[(d)] $|u|_{L^{p(x)}(w,\Om)}^{p^-}-1\leq \rho(u)\leq |u|_{L^{p(x)}(w,\Om)}^{p^+}+1$.

\end{itemize}
\end{proposition}

\begin{proposition}\label{p2.3}
For $u_l,u\in L^{p(x)}(w,\Om)$, the following two statements are equivalent:

\begin{itemize}

\item[(a)] $\lim_{l\to\infty}|u_l-u|_{L^{p(x)}(w,\Om)}=0$;

\item[(b)] $\lim_{l\to\infty}\rho(u_l-u)=0$.

\end{itemize}

\end{proposition}

\begin{proposition}\label{p2.4}
Assume that $p\in C_+(\ov{\Om})$ and $q\in C(\ov{\Om},\R)$ satisfy $pq\in C_+(\ov{\Om})$.Then,
for all $u\in L^{p(x)q(x)}(w,\Om)$, we have
\begin{itemize}

\item[(a)]  $|u|_{L^{p(x)q(x)}(w,\Om)}\geq 1$ $\Longrightarrow$
$|u|^{\frac{(pq)^-}{p^+}}_{L^{p(x)q(x)}(w,\Om)}\leq |u^q|_{L^{p(x)}(w,\Om)}\leq |u|^{\frac{(pq)^+}{p^-}}_{L^{p(x)q(x)}(w,\Om)}$;

\item[(b)] $|u|_{L^{p(x)q(x)}(w,\Om)}\leq 1$ $\Longrightarrow$
$|u|^{\frac{(pq)^+}{p^-}}_{L^{p(x)q(x)}(w,\Om)}\leq |u^q|_{L^{p(x)}(w,\Om)}\leq |u|^{\frac{(pq)^-}{p^+}}_{L^{p(x)q(x)}(w,\Om)}$.

\end{itemize}

As a consequence, we always have
$$
|u^q|_{L^{p(x)}(w,\Om)}\leq 1+|u|^{\frac{(pq)^+}{p^-}}_{L^{p(x)q(x)}(w,\Om)}.
$$
\end{proposition}

The weighted variable exponent Sobolev space $W^{1,p(x)}(w,\Om)$ is defined by
\begin{equation*}
W^{1,p(x)}(w,\Om)=\left\{u\in L^{p(x)}(\Om)\ :\ |\nabla u|\in L^{p(x)}(w, \Om)\right\},
\end{equation*}
equipped with the norm
\begin{equation*}
|u|_{W^{1,p(x)}(w,\Om)}=|u|_{L^{p(x)}(\Om)}+|\na u|_{L^{p(x)}(w,\Om)}.
\end{equation*}
$W_0^{1,p(x)}(w,\Om)$ is defined as the closure of $C_0^{\infty}(\Om)$ in $W^{1,p(x)}(w,\Om)$
with respect to the norm $|u|_{W^{1,p(x)}(w,\Om)}$.

To assure some basic properties of $W^{1,p(x)}(w,\Om)$, we assume that the weight $w$ satisfies
the condition:

\begin{itemize}

\item[(W)] $w\in L^1_{\text{loc}}(\Om)$ and $w^{-s(\cdot)}\in L^1(\Om)$ for some 
$s\in C(\ov{\Om})$ satisfying $s(x)\in \left(\frac{N}{p(x)},\infty\right)\cap \left[\frac{1}{p(x)-1},\infty\right)$
for all $x\in \ov{\Om}$.

\end{itemize}

Proposition \ref{p2.5} was proved in \cite[Proposition 2.4]{ki}.

\begin{proposition}\label{p2.5}
Assume that (W) holds. Then, $W^{1,p(x)}(w,\Om)$ is a separable and reflexive Banach space.

\end{proposition}

For the function $s(x)$ given in (W) and $x\in\ov{\Om}$, let
\begin{equation}\label{2.2}
p_s(x)=\frac{p(x)s(x)}{1+s(x)}\quad  \text{and}\quad
p_s^*(x)=\left\{\begin{array}{ll}
\frac{p(x)s(x)N}{(s(x)+1)N-p(x)s(x)} &\quad \text{if}\ p_s(x)<N,\\\noalign{\medskip}
\infty &\quad  \text{if}\  p_s(x)\geq N.
\end{array}
\right.
\end{equation}

For the proofs of Propositions \ref{p2.6} and \ref{p2.7}, see \cite[Theorem 2.11 and Corollary 2.12]{kwz}, respectively.

\begin{proposition}\label{p2.6}
Assume that (W) holds. If $q\in C_+(\ov{\Om})$ satisfies $q(x)<p_s^*(x)$ on $\ov{\Om}$.
Then, there exists a continuous and compact embedding $W^{1,p(x)}(w,\Om)\hra\hookrightarrow L^{q(x)}(\Om)$.
\end{proposition}

\begin{proposition}\label{p2.7}
Assume that $(W)$ holds. Then
\begin{equation*}
|u|_{L^{p(x)}(\Om)}\leq C |\na u|_{L^{p(x)}(w,\Om)}\quad  \text{for all}\ u\in C_0^{\infty}(\Om),
\end{equation*}
where $C$ is a constant independent of $u$.

\end{proposition}

In view of Proposition \ref{p2.7}, we can define in $W_0^{1,p(x)}(w,p(x))$ an equivalent norm
$$
\|u\|_{W_0^{1,p(x)}(w,\Om)}=|\na u|_{L^{p(x)}(w,\Om)}.
$$

The space setting for problem \eqref{1.1} is the product space
$$
X=W_0^{1,p_1(x)}(w_1,\Om)\times\ldots\times W_0^{1,p_n(x)}(w_n,\Om).
$$
For any $u=(u_1,\ldots,u_n)\in X$, we equip $X$ with the norm
\begin{equation*}
\|u\|=\|u\|_{W_0^{1,p_1(x)}(w_1,\Om)}+\ldots+\|u\|_{W_0^{1,p_n(x)}(w_n,\Om)}.
\end{equation*}
Then, $X$ is a real separable and reflexive Banach space.

\begin{definition}\label{d2.1}
We say that $u=(u_1,\ldots,u_n)\in X$ is a weak solution of system \eqref{1.1} if
\begin{equation*}
\sum_{i=1}^n\int_{\Om}w_i(x)|\na u_i|^{p_i(x)-2}\na u_i\cdot \na v_i dx=
\lm\sum_{i=1}^n\int_{\Om}f_i(x, u_1,\ldots, u_n)v_i dx 
\end{equation*}
for all $v=(v_1,\ldots,v_n)\in X$.
\end{definition}

\setcounter{equation}{0}
\section{Main results}
\label{sec3}

For convenience, for any $t=(t_1,\ldots,t_n)\in\R^n$, let $|t|=\sum_{i=1}^n|t_i|$, and define
\begin{equation*}
F_0(x)=\lim_{\substack{|t|\to 0}}\frac{F(x,t_1,\ldots,t_n)}{\sum_{i=1}^n|t_i|^{p_i^+}}\quad  \text{and}\quad
F_{\infty}(x)=\lim_{\substack{|t|\to\infty}}\frac{F(x,t_1,\ldots,t_n)}{\sum_{i=1}^n|t_i|^{p_i^-}}.
\end{equation*}

We need the following conditions.
 
\begin{itemize}


\item[(H1)] For $i=1,\ldots,n$, $w_i\in L^{1}_{\text{loc}}(\Om)$ and $w_i^{-s_i(\cdot)}\in L^1(\Om)$ for some 
$s_i\in C(\ov{\Om})$ satisfying $s_i(x)\in \left(\frac{N}{p_i(x)},\infty\right)\cap \left[\frac{1}{p_i(x)-1},\infty\right)$
for all $x\in \ov{\Om}$;

\item[(H2)] for each $i=1,\ldots,n$, there exist $C>1$, $\theta_{ij}\in C_+(\ov{\Om})$ with $j=1,\ldots,n$, and $h_i\in L^{\hat{\theta}_{ii}(x)}(\Om)$
such that $\theta_{ij}(x)<p_{j,s_j}^*(x)$ in $\Om$ and
\begin{eqnarray*}
&&|f_i(x,t_1,\ldots,t_n)|\\
&\leq& h_i(x)+C\left(|t_i|^{\theta_{ii}(x)-1}
+\sum_{\substack{j\neq i, j=1}}^n|t_j|^{\frac{\theta_{ij}(x)(\theta_{ii}(x)-1)}{\theta_{ii}(x)}}\right)
\quad  \text{for all}\  (x,t_1,\ldots,t_n)\in\Om\times\R^n,
\end{eqnarray*}
where
\begin{equation*}
p_{j,s_j}(x)=\frac{p_j(x)s_j(x)}{1+s_j(x)}\  \text{and}\
p_{j,s_j}^*(x)=\left\{\begin{array}{ll}
\frac{p_j(x)s_j(x)N}{(s_j(x)+1)N-p_j(x)s_j(x)} &\ \text{if}\ p_{j,s_j}(x)<N,\\\noalign{\medskip}
\infty &\  \text{if}\  p_{j,s_j}(x)\geq N,
\end{array}
\right.
\end{equation*}
with $s_j(x)$ given in (H1);

\item[(H3)] $F\in L^{\infty}\left(\Om\times [0, T]^n\right)$ for any $T>0$;

\item[(H4)] $F_{\infty}(x)\leq 0$ uniformly for $x\in\Om$;

\item[(H5)] there exist a constant $t_0>0$ and a ball $B$ with $\ov{B}\subset \Om$ such that 
$$
\int_BF(x,t_0,\ldots,t_0)dx>0.
$$

\end{itemize}

Now, we state the main results in this paper.

\begin{theorem}\label{t3.1}
Assume that (H1)--(H5) hold.
Then, there exists $\lm_0>0$ such that system \eqref{1.1} has at least one nontrivial weak solution for each $\lm \in (\lm_0,\infty)$.
If, in addition, we further assume that
\begin{itemize}

\item[(H6)] $F_{0}(x)\leq 0$ uniformly for $x\in\Om$;

\item[(H7)] $\max\{p_1^+,\ldots, p_n^+\}<\min\{p_{1,s_1}^*(x),\ldots, p_{n,s_n}^*(x)\}$ in $\ov{\Om}$;

\item[(H8)] $F(t,0,\ldots,0)=0$ in $\Om$.

\end{itemize}
Then, system \eqref{1.1} has at least two distinct nontrivial weak solutions for each $\lm \in (\lm_0,\infty)$.

\end{theorem}

The following two corollaries are direct consequences of Theorem \ref{t3.1}.

\begin{corollary}\label{c3.1}
Assume that (H1)--(H3), (H5), (H7), (H8) hold, and 
\begin{equation}\label{3.1}
F(x,t_1,\ldots,t_n)\leq\sum_{i=1}^n\psi_i(t_i)\quad  \text{for all}\  (x,t_1,\ldots,t_n)\in\Om\times\R_+^n,
\end{equation}
where $\psi_i\in C(\R_+)$, $i=1,\ldots,n$, satisfy the conditions

\begin{itemize}

\item[(H9)] there exists $0<T_1\leq 1$ such that $\psi_i(t)\leq 0$ in $[0,T_1)$;

\item[(H10)] there exist $D\geq 1$, $T_2\geq 1$, and $\al_i\in C_+(\ov{\Om})$ such that $\al_i^+<p_i^-$ and
$\psi_i(t)\leq D |t|^{\al_i(x)}$ in $[T_2,\infty)$.

\end{itemize}
Then, there exists $\lm_0>0$ such that system \eqref{1.1} has at least two distinct nontrivial weak solutions for each $\lm \in (\lm_0,\infty)$.
\end{corollary}

\begin{corollary}\label{c3.2}
Assume that $a,b\in C(\ov{\Om})$ are positive, $w$ satisfies the condition (W), and 
$p, \be,\ga\in C_+(\ov{\Om})$ with $\beta(x)<\gamma(x)<p^-\leq p^+<p_s^*(x)$ in $\ov{\Om}$,
where $p_s^*(x)$ is defined in \eqref{2.2}.
Then,  there exists $\lm_0>0$ such that the problem
\begin{equation}\label{3.2}
\left\{\begin{array}{l}
-\di(w(x)|\na u|^{p(x)-2}\na u)=\lm \left(a(x)u^{\gamma(x)-1}-b(x)u^{\beta(x)-1}\right)\quad  \text{in}\  \Om, \\\noalign{\medskip}
u=0\ \text{on}\  \pom,\ 
\end{array}
\right.
\end{equation}
has at least two distinct nontrivial weak solutions for each $\lm \in (\lm_0,\infty)$.
\end{corollary}

Corollary \ref{c3.3} below establish the existence of at least two nontrivial nonnegative weak solutions for a scalar problem.

\begin{corollary}\label{c3.3}
Assume that $g\in C(\ov{\Om})$ is positive, $w$ satisfies the condition (W) and $w\in L^{\infty}(\Om)$, 
$p\in C_+(\ov{\Om})$ with $p^-\leq p^+<p_s^*(x)$ in $\ov{\Om}$, $h\in C(\R,\R)$ with $th(t)\leq 0$ on $[0,\infty)$, and $h_0=h_{\infty}=0$,
where $p_s^*(x)$ is defined in \eqref{2.2}, $h_0=\lim_{|t|\to 0}\frac{h(t)}{|t|^{p^+-1}}$, and $h_{\infty}=\lim_{|t|\to\infty}\frac{h(t)}{|t|^{p^--1}}$.
Then,  there exists $\lm_0>0$ such that the problem
\begin{equation}\label{3.3+}
\left\{\begin{array}{l}
-\di(w(x)|\na u|^{p(x)-2}\na u)=\lm g(x)h(u)\quad  \text{in}\  \Om, \\\noalign{\medskip}
u=0\ \text{on}\  \pom,\ 
\end{array}
\right.
\end{equation}
has at least two distinct nontrivial nonnegative weak solutions for each $\lm \in (\lm_0,\infty)$.
\end{corollary}

\begin{remark}\label{r3.1}
{\rm 
(a) One example $h$ satisfying the conditions in Corollary \ref{c3.3} is given by
\begin{equation*}
h(t)=\left\{\begin{array}{ll}
-|t|^{\mu(x)-2}t\ &\text{if}\ |t|\leq 1,\\\noalign{\medskip}
-|t|^{\nu(x)-2}t\ &\text{if}\ |t|>1,
\end{array}
\right.
\end{equation*}
where $\mu, \nu\in C_+(\ov{\Om})$ satisfy $\mu^->p^+$ and $\nu^+<p^-$.

(b) Theorem \ref{t3.1} and Corollaries \ref{c3.1}--\ref{c3.3} extend and complement some related results 
in the literature, for example, in \cite{ki, HS1, MR}.
}
\end{remark}

In the reminder of this section, we provide the following example.

\begin{example}\label{e3.1}
{\rm
In system \eqref{1.1}, let $n=2$, $w_1,w_2$ satisfy (H1), $p_1,p_2$ satisfy (H7), and
\begin{equation}\label{3.3}
f_1(x,t_1,t_2)=t_1^{r_1(x)-1}-t_1^{s_1(x)-1}+t_1t_2\quad  \text{and}\quad
f_1(x,t_1,t_2)=\frac{t_1^2}{2}+t_2^{r_2(x)-1}-t_2^{s_2(x)-1}
\end{equation}
for all $(x, t_1, t_2)\in \Om\times\R_+^2$, where $r_i, s_i\in C_+(\ov{\Om})$, $i=1,2$, satisfy 
\begin{equation}\label{3.4}
s_1(x)\leq\min\{r_1(x), 4\}\leq\max\{r_1(x),4\}<p_1^-\quad  \text{in}\ \ov{\Om}
\end{equation}
and
\begin{equation}\label{3.5}
s_2(x)\leq\min\{r_2(x), 2\}\leq\max\{r_2(x), 2\}<p_2^-\quad  \text{in}\ \ov{\Om}.
\end{equation}
Then, we claim that there exists $\lm_0>0$ such that system \eqref{1.1} has at least two distinct nontrivial weak solutions for each $\lm \in (\lm_0,\infty)$.

\medskip

In fact, from \eqref{3.3}--\eqref{3.5}, it is easy to verify that (H2) holds, and there exists a function $F(x, t_1, t_2)$, defined by
\begin{equation}\label{3.6}
F(x,t_1,t_2)=\sum_{i=1}^2\left(\frac{1}{r_i(x)}t_i^{r_i(x)}-\frac{1}{s_i(x)}t_i^{s_i(x)}\right)+\frac12 t_1^2t_2
\quad  \text{for all} \ (x, t_1, t_2)\in \Om\times\R_+^2,
\end{equation}
such that $\nabla F(x,t_1,t_2)=(f_1(x,t_1, t_2), f_2(x, t_1, t_2))$ in $\ov{\Om}\times\R^2_+$. 
Clearly, $F$ satisfies (H3), (H5), and (H8).
Note that $\frac12 t_1^2t_2\leq (t_1^4+t_2^2)/4$. Then, from \eqref{3.6}, we see that
\begin{equation*}
F(x,t_1,t_2)\leq \sum_{i=1}^2\psi_i(t_i)\quad  \text{for all} \ (x, t_1, t_2)\in \Om\times\R_+^2,
\end{equation*}
where
\begin{equation*}
\psi_1(t_1)=\frac{1}{r_1(x)}t_1^{r_1(x)}+\frac14 t_1^4-\frac{1}{s_1(x)}t_1^{s_1(x)}\quad  \text{and}\quad 
\psi_2(t_2)=\frac{1}{r_2(x)}t_2^{r_2(x)}+\frac14 t_2^2-\frac{1}{s_1(x)}t_2^{s_2(x)}.
\end{equation*}
In view of \eqref{3.4} and \eqref{3.5}, $\psi_1$ and $\psi_2$ satisfy (H9) and (H10). We have verified that all the conditions
of Corollary \eqref{c3.1} are satisfied. Hence, the claim readily follows from Corollary \ref{c3.1}.
}
\end{example}

\setcounter{equation}{0}
\section{Proofs of the main results}\label{sec4}

Define the functionals $\Phi, \Psi, I : X\to \R$ by
\begin{equation*}
\Phi(u)=\sum_{i=1}^n\int_{\Om}\frac{w_i(x)}{p_i(x)}|\na u_i|^{p_i(x)}dx,
\end{equation*}
\begin{equation*}
\Psi(u)=\int_{\Om}F(x,u_1,\ldots,u_n)dx,
\end{equation*}
and
\begin{equation*}
I(u)=\Phi(u)-\lm\Psi(u).
\end{equation*}

\begin{lemma}\label{l4.1}
Assume that (H1) and (H2) hold. Then, we have the following:
\begin{itemize}

\item[(a)] $\Phi\in C^1(X,\R)$ with the derivative given by
\begin{equation}\label{4.1}
\la\Phi'(u), v\ra=\sum_{i=1}^n\int_{\Om}w_i(x)|\na u_i|^{p_i(x)-2}\na u_i\cdot\na v_idx
\end{equation}
for all $u=(u_1,\ldots,u_n),v=(v_1,\ldots,v_n)\in X$.

\item[(b)] $\Psi\in C^1(X,\R)$ with the derivative given by
\begin{equation}\label{4.2}
\la\Psi'(u),v\ra=\sum_{i=1}^n\int_{\Om}f_i(x,u_1,\ldots,u_n)v_idx
\end{equation}
for all $u=(u_1,\ldots,u_n),v=(v_1,\ldots,v_n)\in X$. Moreover, $\Psi$ and $\Psi'$ are sequentially weakly continuous.

\item[(c)] $I\in C^1(X,\R)$ with the derivative given by
\begin{align*}
\la I'(u), v\ra =&\la\Phi'(u), v\ra-\lm \la\Psi'(u),v\ra\\
=&\sum_{i=1}^n\int_{\Om}w_i(x)|\na u_i|^{p_i(x)-2}\na u_i\cdot\na v_idx
-\lm \sum_{i=1}^n\int_{\Om}f_i(x,u_1,\ldots,u_n)v_idx
\end{align*}
for all $u=(u_1,\ldots,u_n),v=(v_1,\ldots,v_n)\in X$.

\end{itemize}
\end{lemma}

The proof of the scalar case Lemma \ref{l4.1} is contained in the proofs of \cite[Lemma 3.1]{HS2} and \cite[Proposition 2.9]{HS1}. 
Since much more delicate arguments are needed to deal with the system case, we provide a proof below.
The proof here is motivated by the work in \cite{HS1, HS2}.


\begin{proof}
We first prove part (a).
For any $u=(u_1,\ldots, u_n),v=(v_1,\ldots,v_n)\in X$, $x\in \Om$, and $i=1,\dots,n$, from the mean value theorem in several variables, 
there exists $c\in (0,1)$ such that for any $t\in\R$ with $0<|t|<1$, 
\begin{equation*}
\frac{\frac{w_i(x)}{p_i(x)}|\na  u_i+t\na v_i|^{p_i(x)}-\frac{w_i(x)}{p_i(x)}|\na u_i|^{p_i(x)}}{t}
=w_i(x)|\na u_i+ct \na v_i|^{p_i(x)-2}(\na u_i+ct \na v_i)\cdot \na v_i.
\end{equation*}
Hence,
\begin{equation}\label{4.3}
\lim_{t\to 0}\frac{\frac{w_i(x)}{p_i(x)}|\na  u_i+t\na v_i|^{p_i(x)}-\frac{w_i(x)}{p_i(x)}|\na u_i|^{p_i(x)}}{t}=
w_i(x)|\na u_i|^{p_i(x)-2}\na u_i\cdot \na v_i
\end{equation}
and
\begin{equation}\label{4.4}
\left|\frac{\frac{w_i(x)}{p_i(x)}|\na  u_i+t\na v_i|^{p_i(x)}-\frac{w_i(x)}{p_i(x)}|\na u_i|^{p_i(x)}}{t}\right|
\leq w_i(x)|(\na u_i|+|\na v_i|)^{p_i(x)-1}|\na v_i|.
\end{equation}
Note that $\left|w_i^{\frac{1}{p(x)}}|\na v_i|\right|_{L^{p(x)}}=\|v_i\|_{W_0^{1,p_i(x)}(w_i,\Om)}$.
Then, by Propositions \ref{p2.1} and \ref{p2.2}, we have
\begin{eqnarray*}
&&\int_{\Om}w_i(x)|(\na u_i|+|\na v_i|)^{p_i(x)-1}|\na v_i|dx\\
&\leq& 2\left|w_i^{\frac{1}{\hat{p}_i(x)}}(|\na u_i|+|\na v_i|)^{p_i(x)-1} \right|_{L^{\hat{p}_i(x)}(\Om)}
\left|w_i^{\frac{1}{p(x)}}|\na v_i|\right|_{L^{p_i(x)}(\Om)}\\
&\leq& 2\left[1+\left(\int_{\Om}w_i(x)|(\na u_i|+|\na v_i|)^{p_i(x)}dx\right)^{\frac{1}{\hat{p}_i^-}}\right]
\|v_i\|_{W_0^{1,p_i(x)}(w_i,\Om)}\\
&\leq& 2\left[1+2^{\frac{p_i^+-1}{\hat{p}_i^-}}\left(\int_{\Om}w_i(x)|\left(\na u_i|^{p_i(x)}+|\na v_i|^{p_i(x)}\right)dx\right)^{\frac{1}{\hat{p}_i^-}}\right]
\|v_i\|_{W_0^{1,p_i(x)}(w_i,\Om)}.
\end{eqnarray*} 
Thus, $w_i|\left(\na u_i|+|\na v_i|\right)^{p_i(\cdot)-1}|\na v_i|\in L^1(\Om)$ in view of Proposition \ref{p2.2} and the fact that $u,v\in X$.
Then, from \eqref{4.3}, \eqref{4.4}, and the Lebesgue dominated convergence theorem, we have
\begin{equation*}
\lim_{t\to 0}\int_{\Om}\frac{\frac{w_i(x)}{p_i(x)}|\na  u_i+t\na v_i|^{p_i(x)}-\frac{w_i(x)}{p_i(x)}|\na u_i|^{p_i(x)}}{t}dx=
\int_{\Om}w_i(x)|\na u_i|^{p_i(x)-2}\na u_i\cdot \na v_idx.
\end{equation*}
Hence,
\begin{eqnarray*}
&&\lim_{t\to 0}\frac{\Phi(u+tv)-\Phi(u)}{t}\\
&=&\sum_{t=0}^n\left(\lim_{t\to 0}\int_{\Om}\frac{\frac{w_i(x)}{p_i(x)}|\na  u_i+t\na v_i|^{p_i(x)}-\frac{w_i(x)}{p_i(x)}|\na u_i|^{p_i(x)}}{t}dx\right)\\
&=&\sum_{i=1}^n\int_{\Om}w_i(x)|\na u_i|^{p_i(x)-2}\na u_i\cdot \na v_idx,
\end{eqnarray*}
i.e., $\Phi$ is G\^{a}teaux differentiable and \eqref{4.1} holds.

We now show that $\Phi' : X\to X^*$ is continuous. To this end, let
$u_l=(u_{l1},\ldots,u_{ln}), u=(u_1,\ldots,u_n)\in X$ be such that $u_l\to u$ in $X$ as $l\to\infty$. Then, for $i=1,\ldots,n$, 
from Proposition \ref{p2.3}, we see that
$\lim_{l\to\infty}\int_{\Om}w_i(x)|\na u_{li}-\na u_{i}|^{p_i(x)}dx=0$. 
Then, up to a subsequence, we have
\begin{equation}\label{4.5}
\na u_{li}\to \na u_i\quad \text{a.e. in}\ \Om\ \text{as}\ l\to\infty
\end{equation}
and
\begin{equation}\label{4.6}
w_i(x)|\na u_{li}-\na u_{i}|^{p_i(x)}\leq m_i(x)\quad  \text{a.e. in}\ \Om\ \text{for some}\ m_i\in L^1(\Om).
\end{equation}
Note that
\begin{eqnarray*}
w_i(x)|\na u_{li}|^{p_i(x)}&\leq& w_i(x)\left(|\na u_i|+|\na u_{li}-\na u_i|^{p_i(x)}\right)\\
&\leq& 2^{p_i^+}w_i(x)\left(|\na u_i|^{p_i(x)}+|\na u_{li}-\na u_i|^{p_i(x)}\right).
\end{eqnarray*}
Then, from \eqref{4.6},
\begin{equation}\label{4.7}
w_i(x)|\na u_{li}|^{p_i(x)}\leq 2^{p_i^+}\left(w_i(x)|\na u_i|^{p_i(x)}+m_i(x)\right). 
\end{equation}
For any $v=(v_1,\ldots,v_n)\in X$ with $\|v\|\leq 1$, by Proposition \ref{p2.1}, it follows that 
\begin{eqnarray*}
&&|\la\Phi'(u_l)-\Phi'(u), v\ra|\\
&=&\sum_{i=1}^n\int_{\Om}w_i(x)\left(|\na u_{li}|^{p_i(x)-2}\na u_{li}
-|\na u_i|^{p_i(x)-2}\na u_i\right)\cdot\na v_idx\\
&\leq&2\sum_{i=1}^n\left|w_i^{\frac{1}{\hat{p}_i(x)}}\left||\na u_{li}|^{p_i(x)-2}\na u_{li}-|\na u_i|^{p_i(x)-2}\na u_i\right|\right|_{L^{\hat{p}_i(x)}(\Om)}
\left|w_i^{\frac{1}{p(x)}}|\na v_i|\right|_{L^{p_i(x)}(\Om)}\\
&\leq&2\sum_{i=1}^n\left|w_i^{\frac{1}{\hat{p}_i(x)}}\left||\na u_{li}|^{p_i(x)-2}\na u_{li}-|\na u_i|^{p_i(x)-2}\na u_i\right|\right|_{L^{\hat{p}_i(x)}(\Om)}.
\end{eqnarray*}
Thus,
\begin{equation}\label{4.8}
\|\Phi'(u_l)-\Phi'(u)\|_{X^*}\leq
2\sum_{i=1}^n\left|w_i^{\frac{1}{\hat{p}_i(x)}}\left||\na u_{li}|^{p_i(x)-2}\na u_{li}-|\na u_i|^{p_i(x)-2}\na u_i\right|\right|_{L^{\hat{p}_i(x)}(\Om)}.
\end{equation}
Obviously, we have
\begin{eqnarray*}
&&\int_{\Om}\left|w_i^{\frac{1}{\hat{p}_i(x)}}(x)\left||\na u_{li}|^{p_i(x)-2}\na u_{li}-|\na u_i|^{p_i(x)-2}\na u_i\right|\right|^{\hat{p}_i(x)}dx\nonumber\\
&=&\int_{\Om}w_i(x)\left||\na u_{li}|^{p_i(x)-2}\na u_{li}-|\na u_i|^{p_i(x)-2}\na u_i\right|^{\hat{p}_i(x)}dx.
\end{eqnarray*}
Note from \eqref{4.5} that
\begin{equation*}
w_i(x)\left||\na u_{li}|^{p_i(x)-2}\na u_{li}-|\na u_i|^{p_i(x)-2}\na u_i\right|^{\hat{p}_i(x)}\to 0\quad  \text{a.e. in}\ \Om\ \text{as}\ l\to\infty,
\end{equation*}
and from \eqref{4.7} that
\begin{eqnarray*}
&&w_i(x)\left||\na u_{li}|^{p_i(x)-2}\na u_{li}-|\na u_i|^{p_i(x)-2}\na u_i\right|^{\hat{p}_i(x)}\\
&\leq& 2^{\hat{p}_i^+-1}w_i(x)\left(|\na u_{li}|^{p_i(x)}+|\na u_i|^{p_i(x)}\right)\\
&\leq& 2^{\hat{p}_i^++p_i^+-1}\left(w_i(x)|\na u_{i}|^{p_i(x)}+m_i(x)\right).
\end{eqnarray*}
Since $2^{\hat{p}_i^++p_i^+-1}\left(w_i(\cdot)|\na u_{i}|^{p_i(\cdot)}+m_i(\cdot)\right)\in L^1(\Om)$, from 
the Lebesgue dominated convergence theorem, it follows that
\begin{eqnarray*}
\lim_{l\to\infty}\int_{\Om}\left|w_i^{\frac{1}{\hat{p}_i(x)}}(x)\left||\na u_{li}|^{p_i(x)-2}\na u_{li}
-|\na u_i|^{p_i(x)-2}\na u_i\right|\right|^{\hat{p}_i(x)}dx=0.
\end{eqnarray*}
This, together with Proposition \ref{p2.3}, implies that
\begin{equation*}
\lim_{l\to\infty}\left|w_i^{\frac{1}{\hat{p}_i(x)}}\left||\na u_{li}|^{p_i(x)-2}\na u_{li}-|\na u_i|^{p_i(x)-2}\na u_i\right|\right|_{L^{\hat{p}_i(x)}(\Om)}=0.
\end{equation*}
Thus, from \eqref{4.8}, we have
\begin{equation*}
\|\Phi'(u_l)-\Phi'(u)\|_{X^*}=0.
\end{equation*}
Hence, $\Phi' : X\to X^*$ is continuous, and so $\Phi\in C^1(X,\R)$.
This proves part (a).

Next, we show part (b).
Let $u=(u_1,\ldots, u_n),v=(v_1,\ldots,v_n)\in X$, $x\in \Om$, and $t\in\R$.
Then,
\begin{eqnarray}\label{4.9}
&&\frac{F(x,u_1+tv_1,\ldots,u_n+tv_n)-F(x,u_1,\ldots,u_n)}{t}\nonumber\\
&=&\sum_{i=1}^n\int_0^1f_i(x,u_1+stv_1,\ldots,u_n+stv_n)v_ids.
\end{eqnarray}
From (H2), we see that
\begin{eqnarray}\label{4.10}
&&\sum_{i=1}^n\int_0^1f_i(x,u_1+stv_1,\ldots,u_n+stv_n)v_ids\nonumber\\
&\leq& \sum_{i=1}^n\int_0^1\left[|h_i(x)| |v_i|+C\left(|u_i+stv_i|^{\theta_{ii}(x)-1}|v_i|
+\sum_{\substack{j\neq i, j=1}}^n|u_j+stv_j|^{\frac{\theta_{ij}(x)(\theta_{ii}(x)-1)}{\theta_{ii}(x)}}|v_i|\right)\right]ds\nonumber\\
&\leq& K(x,u,v),
\end{eqnarray}
where
\begin{equation*}
K(x,u,v)=\sum_{i=1}^n\left[|h_i(x)||v_i|+C\left(|u_i+v_i|^{\theta_{ii}(x)-1}|v_i|
+\sum_{\substack{j\neq i, j=1}}^n|u_j+v_j|^{\frac{\theta_{ij}(x)(\theta_{ii}(x)-1)}{\theta_{ii}(x)}}|v_i|\right)\right].
\end{equation*}
Propositions \ref{p2.1} and \ref{p2.4} imply that
\begin{equation*}
\int_{\Om}|h_i(x)||v_i|dx\leq 2|h_i|_{L^{\hat{\theta}_{ii}(x)}(\Om)}|v_i|_{L^{\theta_{ii}(x)}(\Om)}
\end{equation*}
and
\begin{eqnarray*}
\int_{\Om}|u_i+v_i|^{\theta_{ii}(x)-1}|v_i|dx&\leq& 2\left|(|u_i|+|v_i|)^{\ta_{ii}(x)-1}\right|_{L^{\hat{\ta}_{ii}(x)}(\Om)}|v_i|_{L^{\ta_{ii}(x)}(\Om)}\\
&\leq& 2\left[1+\left(|u_i|_{L^{\ta_{ii}(x)}(\Om)}+|v_i|_{L^{\ta_{ii}(x)}(\Om)}\right)^{\frac{\ta_{ii}^+}{\hat{\ta}_{ii}^-}}\right]|v_i|_{L^{\ta_{ii}(x)}(\Om)},
\end{eqnarray*}
and by Young's inequality and Proposition \ref{p2.2}, we have
\begin{eqnarray*}
\int_{\Om}|u_j+v_j|^{\frac{\theta_{ij}(x)(\theta_{ii}(x)-1)}{\theta_{ii}(x)}}|v_i|dx
&\leq& \int_{\Om}\left[\frac{\theta_{ii}(x)-1}{\ta_{ii}(x)}(|u_j|+|v_j|)^{\ta_{ij}(x)}+\frac{1}{\ta_{ii}(x)}|v_i|^{\ta_{ii}(x)}\right]dx\\
&\leq& 2^{\ta_{ii}^+}\int_{\Om}\left(|u_j|^{\ta_{ij}(x)}+|v_j|^{\ta_{ij}(x)}+|v_i|^{\ta_{ii}(x)}\right)dx\\
&\leq&2^{\ta_{ii}^+}\left(|u_j|^{\ta_{ij}^+}_{L^{\ta_{ij}(x)}(\Om)}+|v_j|^{\ta_{ij}^+}_{L^{\ta_{ij}(x)}(\Om)}+|v_i|^{\ta_{ii}^+}_{L^{\ta_{ii}(x)}(\Om)}+3\right).
\end{eqnarray*}
Then, from Proposition \ref{p2.6} and the above estimates, we see that $K(\cdot,u(\cdot),v(\cdot))\in L^1(\Om)$.
Hence, in view of \eqref{4.9} and \eqref{4.10}, from the Lebesgue dominated convergence theorem, it follows that
\begin{eqnarray*}
&&\lim_{t\to 0}\frac{\Psi(u+tv)-\Psi(u)}{t}\\
&=&\lim_{t\to 0}\int_{\Om}\frac{F(x,u_1+tv_1,\ldots,u_n+tv_n)-F(x,u_1,\ldots,u_n)}{t}dx\\
&=&\sum_{i=1}^n\int_{\Om}f_i(x,u_1,\ldots,u_n)v_idx,
\end{eqnarray*} 
i.e., $\Psi$ is G\^{a}teaux differentiable and \eqref{4.2} holds.

We now show that $\Psi' : X\to X^*$ is continuous. To this end, let
$u_l=(u_{l1},\ldots,u_{ln}), u=(u_1,\ldots,u_n)\in X$ be such that $u_l\to u$ in $X$ as $l\to\infty$. Then, for $i, j=1,\ldots,n$, 
from Proposition \ref{p2.6}, we see that $u_{lj}\to u_j$ in $L^{\ta_{ij}(x)}(\Om)$ as $l\to\infty$. Then, up to a subsequence, we obtain that
\begin{equation}\label{4.11}
u_{lj}\to u_j\quad  \text{a.e in}\ \Om\ \text{as}\ l\to\infty 
\end{equation}
and
\begin{equation}\label{4.12}
|u_{lj}(x)|^{\ta_{ij}(x)}\leq k_j(x)\quad  \text{a.e. in}\ \Om \ \text{for some}\ k_j\in L^1(\Om).
\end{equation}
For all $v=(v_1,\ldots,v_n)\in X$ with $\|v\|\leq 1$, from Propositions \ref{p2.1} and \ref{p2.6}, we see that
\begin{eqnarray*}
&&|\la\Psi'(u_l)-\Psi'(u), v\rangle|\\
&\leq& \sum_{i=1}^n\int_{\Om}\left|f_i(x,u_{l1},\ldots,u_{ln})-f_i(x,u_1,\ldots,u_n)\right| |v_i|dx\\
&\leq&2\sum_{i=1}^n\left|f_i(\cdot,u_{l1}(\cdot),\ldots,u_{ln}(\cdot))-f_i(\cdot,u_1(\cdot),\ldots,u_n(\cdot))\right|_{L^{\hat{\ta}_{ii}(x)}(\Om)}|v_i|_{L^{{\ta}_{ii}(x)}(\Om)}\\
&\leq&2C_{\ta_{ii}}\sum_{i=1}^n\left|f_i(\cdot,u_{l1}(\cdot),\ldots,u_{ln}(\cdot))-f_i(\cdot,u_1(\cdot),\ldots,u_n(\cdot))\right|_{L^{\hat{\ta}_{ii}(x)}(\Om)},
\end{eqnarray*}
where $C_{\ta_{ii}}>0$ is the embedding constant of the compact embedding $W_0^{1,p_i(x)}(w_i,\Om)\hra\hra L^{\ta_{ii}(x)}(\Om)$. Hence,
\begin{eqnarray}\label{4.13}
&&\|\Psi'(u_l)-\Psi'(u)\|_{X^*}\nonumber\\
&\leq& 2C_{\ta_{ii}}\sum_{i=1}^n\left|f_i(\cdot,u_{l1}(\cdot),\ldots,u_{ln}(\cdot))-f_i(\cdot,u_1(\cdot),\ldots,u_n(\cdot))\right|_{L^{\hat{\ta}_{ii}(x)}(\Om)}.
\end{eqnarray}
For $i=1,\ldots,n$, from (H1) and \eqref{4.12}, we have
\begin{eqnarray*}
&&\left|f_i(x,u_{l1},\ldots,u_{ln})-f_i(x,u_1,\ldots,u_n)\right|^{\hat{\ta}_{ii}(x)}\\
&\leq& 2^{\hat{\ta}_{ii}^+}\left(|f_i(x,u_{l1},\ldots,u_{ln})|^{\hat{\ta}_{ii}(x)}+|f_i(x,u_1,\ldots,u_n)|^{\hat{\ta}_{ii}(x)}\right)\\
&\leq&2^{2\hat{\ta}_{ii}^+}\left[|h_i(x)|^{\hat{\ta}_{ii}(x)}+C^{\hat{\ta}_{ii}(x)}\left(|u_{li}|^{\theta_{ii}(x)-1}
+\sum_{j\neq i, j=1}^{n}|u_{lj}|^{\frac{\ta_{ij}(x)(\ta_{ii}(x)-1)}{\ta_{ii}(x)}}\right)^{\hat{\ta}_{ii}(x)}\right]\\
&&+2^{2\hat{\ta}_{ii}^+}\left[|h_i(x)|^{\hat{\ta}_{ii}(x)}+C^{\hat{\ta}_{ii}(x)}\left(|u_{i}|^{\theta_{ii}(x)-1}
+\sum_{j\neq i, j=1}^{n}|u_{j}|^{\frac{\ta_{ij}(x)(\ta_{ii}(x)-1)}{\ta_{ii}(x)}}\right)^{\hat{\ta}_{ii}(x)}\right]\\
&\leq& D_1\left(|h_i(x)|^{\hat{\ta}_{ii}(x)}+\sum_{j=1}^{n}|u_{lj}|^{\ta_{ij}(x)}+\sum_{j=1}^{n}|u_{j}|^{\ta_{ij}(x)}\right)\\
&\leq& D_1\left(|h_i(x)|^{\hat{\ta}_{ii}(x)}+\sum_{j=1}^{n}|k_j(x)|^{\ta_{ij}(x)}+\sum_{j=1}^{n}|u_{j}|^{\ta_{ij}(x)}\right):=L(x,u),
\end{eqnarray*}
where $D_1>0$ is some appropriate constant. Clearly, $L(\cdot,u(\cdot))\in L^1(\Om)$. Then, from the Lebesgue dominated convergence theorem and \eqref{4.11}, it follows that
\begin{eqnarray*}
\lim_{l\to\infty}\left|f_i(x,u_{l1}(x),\ldots,u_{ln}(x))-f_i(x,u_1(x),\ldots,u_n(x))\right|^{\hat{\ta}_{ii}(x)} dx=0.
\end{eqnarray*}
Then, by Proposition \ref{p2.3}, we have
\begin{equation*}
\lim_{l\to\infty}\left|f_i(\cdot,u_{l1}(\cdot),\ldots,u_{ln}(\cdot))-f_i(\cdot,u_1(\cdot),\ldots,u_n(\cdot))\right|_{L^{\hat{\ta}_{ii}(x)}(\Om)}=0.
\end{equation*}
Consequently,
\begin{equation*}
\lim_{l\to\infty}\sum_{i=1}^n\left|f_i(\cdot,u_{l1}(\cdot),\ldots,u_{ln}(\cdot))-f_i(\cdot,u_1(\cdot),\ldots,u_n(\cdot))\right|_{L^{\hat{\ta}_{ii}(x)}(\Om)}=0.
\end{equation*}
This, together with \eqref{4.13}, implies that
\begin{equation*}
\lim_{l\to\infty}\|\Psi'(u_l)-\Psi'(u)\|_{X^*}=0.
\end{equation*}
Hence, $\Psi' : X\to X^*$ is continuous, and so $\Psi\in C^1(X,\R)$.

Next, we show that $\Psi$ is sequentially weak continuous.
Let
$u_l=(u_{l1},\ldots,u_{ln}), u=(u_1,\ldots,u_n)\in X$ be such that $u_l\rhp u$ in $X$ as $l\to\infty$. Then, for $i, j=1,\ldots,n$, 
from Proposition \ref{p2.6}, we see that $u_{lj}\to u_j$ in $L^{\ta_{ij}(x)}(\Om)$ as $l\to\infty$. Then, 
in view of Proposition \ref{p2.3} with $w\equiv 1$, up to a subsequence, we have
\begin{equation}\label{4.14}
\left\{\begin{array}{l}
u_{lj}\to u_j\quad  \text{a.e in}\ \Om\ \text{as}\ l\to\infty, \\\noalign{\medskip}
|u_{lj}(x)-u_j|^{\ta_{ij}(x)}\leq g_j(x)\quad  \text{a.e. in}\ \Om \ \text{for some}\ g_j\in L^1(\Om).
\end{array}
\right.
\end{equation}
Then, $F(x,u_{l1}(x),\ldots,u_{ln}(x))\to F(x,u_{1}(x),\ldots,u_{n}(x))$ a.e. in $\Omega$ as $l\to\infty$.
From (H1), it follows that
\begin{eqnarray*}
&&|F(x,u_{l1},\ldots,u_{ln})|\\
&\leq &\left|F(x,0,\ldots,0)+\sum_{i=1}^n\int_0^1f_i(x,su_{l1},\ldots,su_{ln})u_{li}ds\right|\\
&\leq&|F(x,0,\ldots,0)|\\
&&+\sum_{i=1}^n\int_0^1\left[|h_i(x)| |u_{li}|+C\left(|su_{li}|^{\theta_{ii}(x)-1}|u_{li}|
+\sum_{\substack{j\neq i, j=1}}^n|su_{lj}|^{\frac{\theta_{ij}(x)(\theta_{ii}(x)-1)}{\theta_{ii}(x)}}|u_{li}|\right)\right]ds\\
&\leq&|F(x,0,\ldots,0)|+\sum_{i=1}^n\left[|h_i(x)| |u_{li}|+C\left(|u_{li}|^{\theta_{ii}(x)}
+\sum_{\substack{j\neq i, j=1}}^n|u_{lj}|^{\frac{\theta_{ij}(x)(\theta_{ii}(x)-1)}{\theta_{ii}(x)}}|u_{li}|\right)\right].
\end{eqnarray*}
Then, by Young's inequality and \eqref{4.14}, we see that
\begin{eqnarray*}
&&|F(x,u_{l1},\ldots,u_{ln})|\\
&\leq& |F(x,0,\ldots,0)|\\
&&+\sum_{i=1}^n\left[|h_i(x)| |u_{li}|+C\left(|u_{li}|^{\theta_{ii}(x)}
+\sum_{\substack{j\neq i, j=1}}^n\left[\frac{\theta_{ii}(x)-1}{\ta_{ii}(x)}|u_{lj}|^{\ta_{ij}(x)}+\frac{1}{\ta_{ii}(x)}|u_{li}|^{\ta_{ii}(x)}\right]\right)\right].\\
&\leq& |F(x,0,\ldots,0)|+D_2\sum_{i=1}^n\left(|h_i(x)| |u_{li}|+\sum_{j=1}^n|u_{lj}|^{\theta_{ij}(x)}\right).\\
&\leq& |F(x,0,\ldots,0)|+D_2\sum_{i=1}^n\left[|h_i(x)| (|u_{li}-u_i|+|u_i|)+\sum_{j=1}^n(|u_{lj}-u_j|+|u_j|)^{\theta_{ij}(x)}\right]\\
&\leq& |F(x,0,\ldots,0)|+D_2\sum_{i=1}^n\left[|h_i(x)| \left((g_{i}(x))^{\frac{1}{\ta_{ii}(x)}}+|u_i|\right)
+\sum_{j=1}^n\left((g_j(x))^{\frac{1}{\ta_{ij}(x)}}+|u_j|\right)^{\theta_{ij}(x)}\right].\\
&\leq& M(x,u),
\end{eqnarray*}
where $D_2>0$ is some appropriate constant and 
\begin{eqnarray*}
M(x,u)= |F(x,0,\ldots,0)|+D_2\sum_{i=1}^n\left[|h_i(x)| \left((g_i(x))^{\frac{1}{\ta_{ii}(x)}}+|u_i|\right)
+\sum_{j=1}^n2^{\ta_{ij}^+}\left(g_j(x)+|u_j|^{{\ta_{ij}(x)}}\right)\right].
\end{eqnarray*}
By Propositions \ref{p2.1}, we obtain that
\begin{equation*}
\int_{\Om}|h_i(x)|(g_i(x))^{\frac{1}{\ta_{ii}(x)}}dx\leq 2|h_i|_{L^{\hat{\theta}_{ii}(x)}(\Om)}|g_i|_{L^{1}(\Om)}
\end{equation*}
and
\begin{equation*}
\int_{\Om}|h_i(x)||u_i|dx\leq 2|h_i|_{L^{\hat{\theta}_{ii}(x)}(\Om)}|u_i|_{L^{\theta_{ii}(x)}(\Om)}.
\end{equation*}
Then, in view of Proposition \ref{p2.6} and the above estimates, we see that $M(\cdot,u(\cdot))\in L^1(\Om)$.
Thus, by the Lebesgue dominated convergence theorem, we have
\begin{equation*}
\lim_{l\to\infty}\int_{\Om}F(x,u_{l1},\ldots,u_{ln})=\int_{\Om}F(x,u_{1},\ldots,u_{n}),
\end{equation*}
i.e., $\lim_{l\to\infty}\Psi(u_l)=\Psi(u)$.
Hence, $\Psi$ is sequentially weakly continuous.
By a similar argument, we can sow the sequentially weak continuity of $\Psi'$. The details are omitted here.

Finally, part (c) readily follows from parts (a) and (b).
The completes the proof of the lemma.
\end{proof}

\begin{lemma}\label{l4.2}
Assume that (H1) holds. Then, the operator $\Phi'(u): X\to X^*$ is of type $(S_+)$, i.e., if
$u_l=(u_{l1},\ldots,u_{ln}), u=(u_1,\ldots,u_n)\in X$ are such that $u_l\rhp u$ in $X$ and
$\lim\sup_{l\to\infty}\la\Phi'(u_l), u_l-u\ra\leq 0$, then $u_l\to u$ in $X$.
\end{lemma}

The scalar case of Lemma \ref{l4.2} follows from \cite[Theorem 4.1]{le}; see also \cite[Lemma 3.2]{HS2}.
The general case can be proved using an argument similar to that of \cite[Theorem 4.1]{le}.

\begin{remark}\label{r4.1}
{\rm
Under the conditions (H1) and (H2), we have the following observations:

\begin{itemize}

\item[(a)] In view of Lemmas \ref{l4.1} (b) and \ref{l4.2}, $I'=\Phi'-\lm\Psi'$ is the sum of a $(S_+)$ operator and
a sequentially weakly continuous operator. Hence, $I'$ is of type $(S_+)$. 


\item[(b)] $\Phi$ is weakly lower semicontinuous since it is convex. By Lemma \ref{l4.1} (b), $\Psi$ is sequentially weakly
continuous. Thus, $I=\Phi-\lm\Psi$ is weakly lower semicontinuous.

\end{itemize}

}
\end{remark}

\begin{remark}\label{r4.2}
{\rm 
By Definition \ref{d2.1} and Lemma \ref{l4.1} (c), we see that any nontrivial critical points of $I$ are nontrivial weak solutions of system \eqref{1.1}.
}
\end{remark}

Lemma \ref{l4.4} below can be found in \cite{Z}.

\begin{lemma}\label{l4.3}
Let $X$ be a real reflexive Banach space, and let $J$ be a weakly lower semicontinuous functional such that
$
\lim_{\|u\|\to\infty}J(u)=\infty.
$
Then, there exists $u_0\in X$ such that
$
J(u_0)=\inf_{u\in X}J(u).
$ 
 Furthermore, if $J\in C^1(X,R)$, then $J'(u_0)=0$.
\end{lemma}

Recall that a functional $I\in C^1(X,\R)$ is said to satisfy the
Palais--Smale (PS, for short) condition if every sequence $\{u_n\}\subset X$, such that $I(u_n)$ is bounded
and $I'(u_n)\to 0$ as $n\to\infty$, has a convergent subsequence. 
The sequence $\{u_n\}$ is called a PS sequence of $I$.
We now state the following classic mountain pass lemma of Ambrosetti 
and Rabinowitz (see, for example, \cite[Theorem 7.1]{Jabri}).
Below, we denote by $B_r(u)$ the open ball centered at $u \in X$
with radius $r > 0$, $\ov{B}_r(u)$ its closure, and $\partial B_r(u)$ its boundary.

\begin{lemma}\label{l4.4}
Let $(X, \|\cdot\|)$ be a real Banach space and $I\in C^1(X,\R)$. 
Assume that $I$ satisfies the PS condition and there exist $u_0, u_1\in X$ and $\rho>0$ such that
\begin{itemize}

\item[(A1)] $u_1\not\in\ov{B}_{\rho}(u_0)$;

\item[(A2)] $\max\{I(u_0), I(u_1)\}<\inf_{u\in\partial B_{\rho}(u_0)}I(u)$.

\end{itemize}
Then, I possesses a critical value which can be characterized as
$$
c=\inf_{\ga\in\Ga}\max_{s\in [0,1]}I(\ga(s))\geq\inf_{u\in\partial B_{\rho}(u_0)}I(u),
$$
where
$$
\Ga=\left\{\ga\in C([0,1],X)\ :\ \ga(0)=u_0,\ \ga(1)=u_1\right\}.
$$
\end{lemma}

We are now in a position to prove Theorem \ref{t3.1}.

\begin{proof}[Proof of Theorem \ref{t3.1}]
For $i=1,\ldots, n$, by Proposition \ref{p2.6}, there exists $C_i>0$ such that
\begin{equation}\label{4.15}
|v|_{L^{p_i(x)}(\Om)}\leq C_i \|v\|_{W_0^{1,p_i(x)}(w_i,\Om)}\quad  \text{for all}\ v\in W_0^{1,p_i(x)}(w_i,\Om).
\end{equation}
For any fixed $\lm>0$, choose $\ep=\ep(\lm)>0$ small enough so that
\begin{equation}\label{4.16}
\frac{1}{p_i^+}-\ep\lm C_i>0\quad  \text{for}\ i=1,\ldots,n.
\end{equation}
From (H4), there exists $T=T(\ep)>0$ such that
\begin{equation*}
F(x,t_1,\ldots,t_n)\leq \ep\sum_{i=1}^n|t_i|^{p_i^-}\quad  \text{for all}\ x\in\Om\ \text{and}\
t=(t_1,\ldots,t_n)\in\R^n \ \text{with} \ |t|>T.
\end{equation*}
This, together with (H3), implies that
\begin{equation}\label{4.17}
F(x,t_1,\ldots,t_n)\leq C(\ep)+\ep\sum_{i=1}^n|t_i|^{p_i^-}\quad  \text{for all}\ x\in\Om\ \text{and}\
t=(t_1,\ldots,t_n)\in\R^n.
\end{equation}
For any $u=(u_1,\ldots, u_n)\in X$, from Proposition \ref{p2.2} (d), \eqref{4.15}, and \eqref{4.17}, it follows that
\begin{eqnarray*}
I(u)&\geq& \sum_{i=1}^n\frac{1}{p_i^{+}}\int_{\Om}w_i(x)|\na u_i|^{p_i(x)}dx-\lm\int_{\Om}\left(C(\ep)+\ep\sum_{i=1}^n|u_i|^{p_i^-}\right)dx\\
&\geq &\sum_{i=1}^n\left[\frac{1}{p_i^{+}}\left(\|u_i\|_{W_0^{1,p_i(x)}(w_i,\Om)}^{p_i^-}-1\right)
-\ep\lm |u_i|^{p_i^-}_{L^{p_i(x)}(\Om)}\right]
-\lm C(\ep)\meas(\Om)\\
&\geq &\sum_{i=1}^n\left[\left(\frac{1}{p_i^{+}}-\ep\lm C_i\right)\|u_i\|_{W_0^{1,p_i(x)}(w_i,\Om)}^{p_i^-}
-\frac{1}{p_i^+}\right]-\lm C(\ep)\meas(\Om).
\end{eqnarray*}
Then, in view of \eqref{4.16}, $I(u)\to\infty$ as $\|u\|\to\infty$, i.e., $I$ is coercice.
By Remark \ref{r4.1} (b), $I$ is weakly lower semicontinuous.
Now, Lemma \ref{l4.3} implies that $I$ has a global minimizer
$u^{1}=(u^1_1,\ldots, u_n^1)\in X$ and $I'(u^1)=0$.

Now, we show that there exists $u=(u_1,\ldots,u_n)\in X$ such that $I(u)<0$ for large $\lm$.
For any $\ep>0$, let $B_{\ep}=\{x\in\Om \ :\ \text{dist}(x,B)\leq\ep\}$, where $B$ is the ball given in (H5).
Let $\ep>0$ be sufficiently small so that $\ov{B_{\ep}}\subset \Om$.
There exists $v_{\ep}\in C_c^1(\Om)$ such that $0\leq v_{\ep}(x)\leq t_0$ on $\Om$ and
\begin{equation*}
v_{\ep}(x)=\left\{\begin{array}{ll}
t_0, &\  x\in B,\\\noalign{\medskip}
0, &\ x\in\Om\setminus B_{\ep},
\end{array}
\right.
\end{equation*}
where $t_0$ is given in (H5).
Let $u_{\ep}(x)=(v_{\ep}(x),\ldots,v_{\ep}(x))$. Then, $u_{\ep}\in X$ and
\begin{eqnarray*}
I(u_{\ep})&=&\Phi(u_{\ep})-\lm\int_{\Om}F(x,v_{\ep},\ldots, v_{\ep})dx\\
&=&\Phi(u_{\ep})-\lm\int_{B}F(x,t_0,\ldots, t_0)dx-\lm\int_{B_{\ep}\setminus B}F(x,v_{\ep},\ldots, v_{\ep})dx\\
&\leq&\Phi(u_{\ep})-\lm\int_{B}F(x,t_0,\ldots, t_0)dx+\lm |F|_{L^{\infty}\left(\Om\times [-t_0, t_0]^n\right)}\meas(B_{\ep}\setminus B).
\end{eqnarray*}
Choose $\ep>0$ small enough so that
\begin{equation*}
|F|_{L^{\infty}\left(\Om\times [0, t_0]^n\right)}\meas(B_{\ep_0}\setminus B)\leq \frac12 \int_{B}F(x,t_0,\ldots, t_0)dx.
\end{equation*}
Then,
\begin{eqnarray*}
I(u_{\ep_0})\leq\Phi(u_{\ep_0})-\frac{\lm}{2}\int_{B}F(x,t_0,\ldots, t_0)dx<0
\end{eqnarray*}
if $\lm>\lm_0$, where
$$
\lm_0=\frac{2\Phi(u_{\ep_0})}{\int_{B}F(x,t_0,\ldots, t_0)dx}.
$$
Thus, $I(u^1)<0$ for any $\lm>\lm_0$, and so $u^1$ is now nontrivial. Now, by Remark \ref{r4.2}, 
we see that $u^1$ is a nontrivial weak solution of system \eqref{1.1} for all $\lm \in (\lm_0,\infty)$.

Now, we further assume that (H6)--(H8) hold.
We show that system \eqref{1.1} has a second nontrivial weak solution for any $\lm \in(\lm_0,\infty)$.
Recall that $I$ is coercive. Then, $I$ satisfies the PS condition since $I'$ is of type $(S_+)$ by Remark \ref{r4.1} (a).
Below, we show that the conditions $(A1)$ and $(A2)$ of Lemma \ref{l4.4} are satisfied.
For convenience, let $\ov{p}=\max\{p_1^+,\ldots, p_n^+\}$.
From (H7), there exists a
constant $q$ such that 
\begin{equation}\label{4.18}
\ov{p}<q<\min\{p_{1,s_1}^*(x),\ldots, p_{n,s_n}^*(x)\}\quad  \text{on}\ \ov{\Om}.
\end{equation}
Then, for $i=1,\ldots,n$, by Proposition \ref{p2.6}, there exist
$D_3>1$ and  $D_4>1$ such that
\begin{equation}\label{4.19}
|v|_{L^{p_i^+}(\Om)}\leq D_3\|v\|_{W^{1,p_i(x)}_0(w_i,\Om)}\quad \text{for all}\ v\in W^{1,p_i(x)}_0(w_i,\Om)
\end{equation}
and
\begin{equation}\label{4.20}
|v|_{L^{q}(\Om)}\leq D_4\|v\|_{W^{1,p_i(x)}_0(w_i,\Om)}\quad \text{for all}\ v\in W^{1,p_i(x)}_0(w_i,\Om).
\end{equation}
For $\de=\left(2\lm\ov{p}D_3^{\ov{p}}\right)^{-1}$, from (H4), and (H6), there exist $k_1>1$
and $k_2>0$ such that
\begin{equation*}
F(x,t_1,\ldots,t_n)\leq \de\sum_{i=1}^n|t_i|^{p_i^-}\leq\de\sum_{i=1}^n|t_i|^{q}\
\text{for all}\ x\in\Om\ \text{and}\ (t_1,\ldots,t_n)\in\R^n\ \text{with}\ |t_i|>k_1
\end{equation*}
and
\begin{equation*}
F(x,t_1,\ldots,t_n)\leq \de\sum_{i=1}^n|t_i|^{p_i^+}\quad
\text{for all}\ x\in\Om\ \text{and}\ (t_1,\ldots,t_n)\in\R^n\ \text{with}\ |t_i|<k_2.
\end{equation*}
Thus, in view of (H3), we see that there exists $D_5>0$ such that
\begin{equation}\label{4.21}
F(x,t_1,\ldots,t_n)\leq \de\sum_{i=1}^n|t_i|^{p_i^+}+D_5\sum_{i=1}^n|t_i|^{q}\quad
\text{for all}\ x\in\Om\ \text{and}\ (t_1,\ldots,t_n)\in\R^n.
\end{equation}
For $u=(u_1,\ldots,u_n)\in X$ with $\|u\|\leq 1$, from Proposition \ref{p2.2} (c) with $w(x)\equiv 1$ and
\eqref{4.19}--\eqref{4.21}, we have
\begin{eqnarray*}
I(u)&\geq&\frac{1}{\ov{p}}\sum_{i=1}^nw_i(x)|\na u_i|^{p_i(x)}dx
-\lm\de\sum_{i=1}^n\int_{\Om}|u_i|^{p_i^+}dx-\lm D_5\sum_{i=1}^n\int_{\Om}|u_i|^{q}dx\\
&\geq&\frac{1}{\ov{p}}\sum_{i=1}^n\|u_i\|_{W^{1,p_i(x)}_0(w_i,\Om)}^{p_i^+}-\lm\de\sum_{i=1}^nD_3^{p_i^+}\|u_i\|_{W^{1,p_i(x)}_0(w_i,\Om)}^{p_i^+}\\
&&-\lm D_5\sum_{i=1}^nD_4^{q}\|u_i\|_{W^{1,p_i(x)}_0(w_i,\Om)}^{q}\\
&\geq&\frac{1}{\ov{p}}\sum_{i=1}^n\|u_i\|_{W^{1,p_i(x)}_0(w_i,\Om)}^{{p}_i^+}-\lm\de D_3^{\ov{p}}\sum_{i=1}^n\|u_i\|_{W^{1,p_i(x)}_0(w_i,\Om)}^{{p}_i^+}\\
&&-\lm D_4^{{q}}D_5\sum_{i=1}^n\|u_i\|_{W^{1,p_i(x)}_0(w_i,\Om)}^{{q}}\\
&=&\frac{1}{2\ov{p}}\sum_{i=1}^n\|u_i\|_{W^{1,p_i(x)}_0(w_i,\Om)}^{{p}_i^+}-\lm D_4^{{q}}D_5\sum_{i=1}^n\|u_i\|_{W^{1,p_i(x)}_0(w_i,\Om)}^{{q}}\\
&\geq&\frac{1}{2\ov{p}}\sum_{i=1}^n\|u_i\|_{W^{1,p_i(x)}_0(w_i,\Om)}^{\ov{p}}-\lm D_4^{{q}}D_5\sum_{i=1}^n\|u_i\|_{W^{1,p_i(x)}_0(w_i,\Om)}^{{q}}.
\end{eqnarray*}
Recall the well-known inequalities
\begin{equation}\label{4.22}
(t_1+\ldots+t_n)^r\leq 2^{n(r-1)}(t_1^r+\ldots+t_n^r)\quad  \text{and}\quad
(t_1+\ldots+t_n)^r\geq t_1^r+\ldots+t_n^r
\end{equation}
for any $r\geq 1$ and $t_i\geq 0$, $i=1,\ldots, n$. Thus, 
\begin{eqnarray*}
I(u)&\geq&\frac{1}{2^{n(r-1)+1}\ov{p}}\left(\sum_{i=1}^n\|u_i\|_{W^{1,p_i(x)}_0(w_i,\Om)}\right)^{\ov{p}}
-\lm D_4^{{q}}D_5\left(\sum_{i=1}^n\|u_i\|_{W^{1,p_i(x)}_0(w_i,\Om)}\right)^{{q}}\\
&=&\frac{1}{2^{n(r-1)+1}\ov{p}}\|u\|^{\ov{p}}-\lm D_4^{{q}}D_5\|u\|^{{q}}.
\end{eqnarray*}
Note that $\ov{p}<q$. Then, if we let
$$
0<\rho<\min\left\{1, \   \|u^1\|, \  \left(\frac{1}{2^{n(r-1)+1}\lm\ov{p}D_4^{q}D_5}\right)^{\frac{1}{q-\ov{p}}}\right\},
$$
it is clear that $I(u)>0$ for all $u\in \ov{B}_{\rho}(u_0)\setminus\{u_0\}$ and
$$
I(u)\geq \kappa:=\frac{1}{2^{n(r-1)+1}\ov{p}}\rho^{\ov{p}}-\lm D_4^{{q}}D_5\rho^q>0\quad \text{for all}\ u\in\partial B_{\rho}(u_0),
$$
where $u_0=(0,\ldots,0)$. In view of (H8), we have $I(u_0)=0$. Note also that $I(u^1)<0$. Then, 
the conditions (A1) and (A2) of Lemma \ref{l4.4}, with $u_1=u^1$, are satisfied. Hence, by Lemma \ref{l4.5},
$I$ has a critical point $u^2\in X$ with $I(u^2)\geq \kappa>0$.
Clearly, $u^2(x)\not\equiv u^1(x)$ and $u^2$ is nontrivial in $\Om$. Thus,
from Remark \ref{r4.2}, $u^2$ is a second nontrivial solution of system \eqref{1.1}.
This completes the proof of the theorem.
\end{proof}

\begin{proof}[Proof of Corollary \ref{c3.1}]
From \eqref{3.1}, (H9), and (H10), it is easy to see that (H4) and (H6) of Theorem \ref{t3.1} hold. Then, all the 
conditions of Theorem \ref{t3.1} are satisfied. Hence, the conclusion follows from Theorem \ref{t3.1}.
\end{proof}

\begin{proof}[Proof of Corollary \ref{c3.2}]
With $n=1$, $f(x,t)=a(x)t^{\ga(x)-1}-b(x)t^{\beta(x)-1}$, and
$F(x,t)=\frac{a(x)}{\gamma(x)t^{\ga(x)}}-\frac{b(x)}{\be(x)t^{\be(x)}}$, it is easy to verify that
all the conditions of Theorem \ref{t3.1} are satisfied. The conclusion then follows from Theorem \ref{t3.1}.
\end{proof} 

For $u\in W_0^{1, p(x)}(w, \Om)$, denote $u_{-}=-\min\{u_i, 0\}$.
Lemma \ref{l4.5} below follows from \cite[Lemma 4.1]{HS2} and \cite[Theorem 7.6]{gt}.

\begin{lemma}\label{l4.5}
Let $w$ be given as in Corollary \ref{c3.3}. Then, $u_-\in W_0^{1, p(x)}(w, \Om)$ and 
$$
\na u_-=\left\{\begin{array}{ll}
\na u\ & \text{if}\ u<0,\\\noalign{\medskip}
0\ & \text{if}\  u\geq 0.
\end{array}
\right.
$$
\end{lemma}

\begin{proof}[Proof of Corollary \ref{c3.3}]
With $n=1$, $f(x,t)=g(x)h(t)$, and
$F(x,t)=g(x)\int_0^th(s)ds$, it is easy to verify that
all the conditions of Theorem \ref{t3.1} are satisfied. Then, by Theorem \ref{t3.1}, there exists $\lm_0>0$ such that problem \eqref{3.3+} has two distinct
nontrivial weak solutions $u^1$ and $u^2$ for all $\lm\in (\lm_0, \infty)$. In the following, we show they are nonnegative.

In fact, for $i=1,2$, in view of Lemmas \ref{l4.1} (c), \ref{l4.5}, and the condition that $th(t)\leq 0$ on $[0,\infty)$, we have
\begin{eqnarray*}
0&=&\int_{\Om}w(x)|\na u^i|^{p(x)-2}\na u^i\cdot\na u^i_{-}dx-\lm \int_{\Om}g(x)h(u^i)u^i_{-}dx\\
&=&\int_{\Om}w(x)|\na u_-^i|^{p(x)-2}\na u_-^i\cdot\na u^i_{-}dx-\lm \int_{\Om}g(x)h(u_-^i)u^i_{-}dx\\
&\geq&\int_{\Om}w_i(x)|\na u_{i-}|^{p_i(x)}dx.
\end{eqnarray*}
Thus, $u_i\geq 0$ in $\Om$.
This completes the proof of the corollary.
\end{proof}

\end{document}